\documentclass[11pt]{amsart}
\usepackage{amsmath}%
\usepackage{amsfonts}%
\usepackage{amssymb}%
\usepackage{graphicx}
\usepackage{enumerate}
\usepackage{fullpage}
\newtheorem{theorem}{Theorem}[section]
\newtheorem{lemma}[theorem]{Lemma}

\newtheorem{problem}[theorem]{Problem}
\newtheorem{corollary}[theorem]{Corollary}
\newtheorem{proposition}[theorem]{Proposition}

\newtheorem{conjecture}[theorem]{Conjecture}

\newtheorem{construction}[theorem]{Construction}

\def\K{\mathcal{K}}

\def\d{\delta}
\def\eps{\varepsilon}

\def\deg{{\rm deg}}
\def\a{\alpha}
\def\b{\beta}
\def\g{\gamma}
\def\P{\mathcal{P}}

\def\ie{{\em i.e.}}
\def\eg{{\em e.g.}}
\def\ex{\mbox{\rm ex}}

\def\PM{\textbf{PM}}
\def\iP{\mathbf{i}_{\mathcal{P}}}

\include{epsf}

\begin{document}

\title{Recent advances on Dirac-type problems for hypergraphs}
\thanks{
The author is partially supported by NSF grant DMS-1400073.}
\author{Yi Zhao}
\address{Department of Mathematics and Statistics, Georgia State University, Atlanta, GA 30303}
\email{yzhao6@gsu.edu}


\begin{abstract}
A fundamental question in graph theory is to establish conditions that ensure a graph contains certain spanning subgraphs. Two well-known examples are Tutte's theorem on perfect matchings and Dirac's theorem on Hamilton cycles. Generalizations of Dirac's theorem, and related matching and packing problems for hypergraphs, have received much attention in recent years. New tools such as the absorbing method and regularity method have helped produce many new results, and yet some fundamental problems in the area remain unsolved. We survey recent developments on Dirac-type problems along with the methods involved, and highlight some open problems.
\end{abstract}

\maketitle

Given two (hyper)graphs $F$ and $H$, which conditions guarantee $H$ contains $F$ as a subgraph?
When $|V(F)|= |V(H)|$, the decision problem of whether $H$ contains $F$ is often NP-complete, \eg, deciding if a graph $H$ contains a Hamilton cycle is a well-known NP-complete problem. Therefore it is natural to look for sufficient conditions for such problems. A classical result of Dirac \cite{Dirac} states that every graph on $n\ge 3$ vertices with minimum degree $n/2$ contains a Hamilton cycle. Problems that relate the minimum degree (in general, minimum $d$-degree in hypergraphs) to the structure of the (hyper)graphs are often referred to as \emph{Dirac-type problems}. The Dirac-type problems for hypergraphs have received much attention in recent years.
In this survey we concentrate on three such problems: matching problems (Section~1), packing problems (Section~2), and Hamilton cycles (Section~3). Many problems in this survey were already considered in the survey of R\"odl and Ruci\'nski \cite{RR-survey}. However, since this is a fast-growing area, there are new developments in the last few years and we will emphasize these new advances.
Since we only consider Dirac-type problems, we do not discuss matching, packing, or Hamilton cycles in random or quasi-random hypergraphs. We also omit corresponding results in graphs and digraphs. Many results that we omit can be found in other surveys, \eg, K\"uhn and Osthus \cite{KuOs-survey, KOsurvey12, KuOs14ICM}, and Gould \cite{Gou03, Gou14}.

\section{Matching problems}
Given $k\ge 2$, a \emph{$k$-uniform hypergraph} (\emph{$k$-graph}) consists of a vertex set $V$ and an edge set $E$, where each edge is a $k$-element subset ($k$-subset) of $V$. Thus a $2$-graph is simply a graph. In this survey a \emph{hypergraph} refers to a $k$-graph with $k\ge 3$. Given a $k$-graph $H$ with $k\ge 2$, a \emph{matching of size $s$} is a collection of $s$ disjoint edges; a \emph{perfect matching} is a matching that covers the vertex set of $H$ (thus it is necessary that $k$ divides $|V(H)|$). Many open problems in combinatorics can be formulated as a problem of finding perfect matchings in hypergraphs, \eg, Ryser's conjecture that every Latin square of odd order has a transversal, and the existence of combinatorial designs (recently solved by Keevash \cite{Ke-design}).

A well-known result of Tutte \cite{Tu47} characterized all the graphs with perfect matchings and there are efficient algorithms (e.g., Edmond's algorithm \cite{Edmonds}) that determine if a graph has a perfect matching. However, deciding if a 3-partite 3-graph contains a perfect matching is among the first 21 NP-complete problems given by Karp~\cite{Karp}. Therefore it is natural to look for sufficient conditions that guarantee a perfect matching.

\subsection{Degree Conditions for Perfect Matchings}
\label{ss:1}
There are multiple ways to define degrees in hypergraphs. Given a $k$-graph $H$ with a set $S$ of $d$ vertices, where $0 \leq d \leq k-1$, the \emph{degree} of $S$, denoted by $\deg_H(S)$ or simply $\deg(S)$, is the number of edges containing $S$. The \emph{minimum $d$-degree $\delta _{d}
(H)$} of $H$ is the minimum of $\deg(S)$ over all $d$-subsets $S$ of $V(H)$. Hence $\delta_0(H)= e(H)$ is the number of edges in $H$.
We refer to  $\delta _1 (H)$ as the \emph{minimum vertex degree} of $H$ and  $\delta _{k-1} (H)$ as the \emph{minimum codegree} of $H$. The simple monotonicity
\[
\frac{\d_0(H)}{\binom{n}{k}} \ge \frac{\d_1(H)}{\binom{n-1}{k-1}} \ge \cdots \ge \frac{\d_{k-1}(H)}{n-k+1}
\]
suggests that a codegree condition is stronger than other degree conditions.
Bollob\'as, Daykin and Erd\H{o}s \cite{BDE} first related the minimum (vertex) degree to the existence of a large (but far from perfect) matching in $k$-graphs. Daykin and H\"aggkvist \cite{DaHa} extended this result
by showing that every $k$-graph with $\d_1(H)\ge (1 - 1/k)\binom{n-1}{k-1}$ contains a perfect matching.

Given integers $d<k\le n$ such that $k$ divides $n$, define the \emph{minimum $d$-degree threshold} $\mathbf{m_{d}(k, n)}$ as the smallest integer $m$ such that every $k$-graph $H$  on $n$ vertices with $\delta_d(H)\ge m$ contains a prefect matching.
A simple greedy argument shows that $m_1(2, n)=n/2$ for all $n\in 2\mathbb{N}$. Given $k\ge 3$, a result of R\"odl, Ruci\'nski and Szemer\'edi \cite{RRS06} on Hamilton cycles implies that $m_{k-1}(k, n)\le n/2 + o(n)$. 
K\"uhn and Osthus \cite{KO06mat} sharpened this bound to $m_{k-1}(k, n)\le n/2 + 3k^2 \sqrt{n\log n}$ by reducing the problem to the one for $k$-partite $k$-graphs. R\"odl, Ruci\'nski and Szemer\'edi \cite{RRS06mat} improved it further to $m_{k-1}(k, n)\le n/2 + O(\log n)$ by using the absorbing method.  R\"odl, Ruci\'nski, Schacht, and Szemer\'edi \cite{RRSS08} found a simple proof of $m_{k-1}(k, n)\le n/2 + k/4$.
Finally R\"odl, Ruci\'nski and Szemer\'edi \cite{RRS09} determined $m_{k-1}(k, n)$ exactly for all $k\ge 3$ and sufficiently large $n$ (again by the absorbing method).
In order to state this and later results, let us describe a class of extremal configurations that are usually referred to as \emph{divisibility barriers}.
\begin{construction}\cite{TrZh12}
\label{cons:md}
Define $\mathcal{H}_{\rm ext} (n,k)$ to be the family of all $k$-graphs $H=(V, E)$, in which there is a partition of $V$ into two parts $A, B$ and  $i\in \{0, 1\}$ such that $|A|\ne i |V| /k \mod 2$ and $|e\cap A| = i \mod 2$ for all edges $e\in E$.
\end{construction}
It is easy to see that no hypergraph $H\in \mathcal{H}_{\text{ext}} (n,k)$ contains a perfect matching. Indeed, suppose $H$ contains a perfect matching $M$, then $|A|= \sum_{e\in M} |e\cap A| = i |V|/ k \mod 2$, contradicting the definition of $H$.

Define $\delta (n,k, d)$ to be the maximum of the minimum $d$-degrees among all the hypergraphs in  $\mathcal H_{\text{ext}} (n,k)$ and note that $m_{d}(k, n)> \delta(n, k, d)$. It is easy to see that
\begin{equation} \label{eq:nk1}
\delta(n, k, k-1) = \left\{\begin{array}{ll}
{n}/{2} - k + 2 & \text{if $k/2$ is even and $n/k$ is odd}\\
{n}/{2} - k + {3}/{2} & \text{if $k$ is odd and $(n-1)/{2}$ is odd}\\
{n}/{2} - k + {1}/{2} & \text{if $k$ is odd and $(n-1)/{2}$ is even} \\
{n}/{2} - k + 1 & \text{otherwise.}
\end{array} \right.
\end{equation}
\begin{theorem}\cite{RRS09}
\label{thm:RRS}
For $k\ge 3$, $m_{k-1}(k, n)=\delta(n, k, k-1) + 1$ for sufficiently large $n$.
\end{theorem}

With more case analysis, one may determine $\delta (n,k, k-2)$, \eg, it is shown in \cite{TrZh12} that
\[
\delta(n, 4, 2) \le \frac{n^2}{4}-\frac{5n}{4} - \frac{\sqrt{n-3}}{2}+\frac{3}{2}
\]
and equality holds for infinitely many $n$. In general, $\delta (n,k, d)= (1/2+o(1))\binom{n- d}{k- d}$ for any fixed $k>d$ but the general formula of $\delta (n,k, d)$ is unknown -- this is related to the open problem of finding the minima of binary Krawtchouk polynomials. Nevertheless, Treglown and the author \cite{TrZh12, TrZh13} determined $m_{d}(k, n)$ in terms of $\delta(n, k, d)$ for all $d\ge k/2$. 
\begin{theorem} \cite{TrZh12, TrZh13}
\label{thm:TZ12}
Let $k\ge 3$ and $d \geq k/2$. Then $m_{d}(k, n)=\delta(n, k, d) + 1$ for sufficiently large $n$.
\end{theorem}
Previously Pikhurko~\cite{Pik} showed that $m_d(k, n)= (1/2+ o(1))\binom{n-d}{k-d}$ for all $d \geq k/2$.
Independently Czygrinow and Kamat \cite{CzKa} determined $m_2(4, n)$ for sufficiently large $n$.

Another class of extremal constructions are known as \emph{space barriers}.
\begin{construction}\label{cons:ms}
Given $s, k, n\in \mathbb{N}$ such that $s\le \lceil n/k \rceil$ ($k$ may not divide $n$),
let $H^0_s(n,k)$ be the $k$-graph on $n$ vertices whose vertex set is partitioned into two parts $A$ and $B$ such that $|A|= s-1$, and whose edge set consists of all those edges with at least one vertex in $A$. 
When $k$ divides $n$, let $H^0(n, k) := H^0_{n/k}(n, k)$.
\end{construction}
Since each edge contains at least one vertex from $A$ and $|A|< s$, $H^0_s(n,k)$ contains no matching of size $s$; in particular, $H^0(n,k)$ contains no prefect matching. Note that
\[
\delta _{d} (H^0(n,k)) =  \binom{n-d}{k-d}-\binom{(1-1/k)n -d +1}{k-d}\approx \left(1-\left(\frac{k-1}{k}\right) ^{k-d} \right ) \binom{n-d}{k-d}.
\]
H\`an, Person and Schacht~\cite{HPS} proved that $m_1(3, n) = (5/9 + o(1))n\approx  \delta _{1} (H^0(n,3))$. Khan~\cite{khan1} and independently K\"uhn, Osthus and Treglown~\cite{KOT} obtained that $m_1(3, n) = \delta _{1} (H^0(n,3)) +1$ for sufficiently large $n$. Khan~\cite{khan2} also proved that $m_1(4, n) = \delta _{1} (H^0(n,4)) +1$ for sufficiently large $n$. Alon, Frankl, Huang, R\"odl, Ruci\'nski and Sudakov~\cite{AFHRRS}  determined $m_{d}(k, n)$ asymptotically for all $d\ge k-4$, including the new cases when $(k, d)= (5, 1)$, $(5, 2)$, $(6, 2)$, and  $(7, 3)$. 
Very recently Treglown and the author \cite{TrZh15} determined $m_2(5, n)$ and $m_3(7, n)$ exactly for sufficiently large $n$.

All these results point to the following conjecture, whose  asymptotic version \eqref{eq:md} has appeared earlier, \eg, \cite{HPS,KuOs-survey}.

\begin{conjecture}\cite{TrZh15}
\label{generalconj}
Let $k, d \in \mathbb N$ such that $d \leq k-1$. Then for sufficiently large $n\in k\mathbb{N}$,
\[
m_{d}(k, n)= \max \left \{ \delta (n,k,d), \ \binom{n-d}{k-d}-\binom{(1-1/k)n -d +1}{k-d}\right \} + 1.
\]
In particular,
\begin{equation}\label{eq:md}
m_{d}(k, n)= \left(\max \left \{ \frac12, \ 1-\left(\frac{k-1}{k}\right)^{k-d} \right\} + o(1) \right) \binom{n-d}{k-d}.
\end{equation}
\end{conjecture}

Note that for all $1\leq d \leq k-1$,
\[
\left( \frac{k-1}{k} \right) ^{k-d} < \left ( \frac{1}{e} \right)^{1-\frac{d}{k}} \ \ \text{ and } \ \
1- \left (\frac{k-1}{k} \right) ^{k \ln 2} \rightarrow \frac{1}{2} \ \ \text{
as } \ \ k \rightarrow \infty,
\]
where $ln$ denotes the natural logarithm.
Thus, for $1\ll k\ll n$, if $d$ is significantly bigger than $(1-\ln 2)k \approx 0.307k$, then  $\delta(n,k, d) > \binom{n-d}{k-d}-\binom{(1-1/k)n -d +1}{k-d}$.
On the other hand, if $d$ is smaller than $(1-\ln 2)k$ then $\delta(n,k, d) < \binom{n-d}{k-d}-\binom{(1-1/k)n -d +1}{k-d}$ for sufficiently large $n$.

Other than the aforementioned results, no other asymptotic or exact value of $m_{d}(k, n)$ is known. When $k\ge 3$ and $1\le d< k/2$,  H\`an, Person and Schacht~\cite{HPS}  gave a general bound:  $m_{d}(k, n)\le ((k-d)/k + o(1)) \binom{n-d}{k-d}$. This was improved by Markstr\"om and Ruci\'nski \cite{MaRu} to  $m_{d}(k, n)\le ((k-d)/k - 1/k^{k-d}+ o(1)) \binom{n-d}{k-d}$ and by K\"uhn, Osthus and Townsend~\cite{KOTo} to
\[
m_{d}(k, n)\le \left( \frac{k-d}k - \frac{k-d-1}{k^{k-d}}+ o(1) \right) \binom{n-d}{k-d}.
\]

\medskip
Let us discuss proof techniques. Most aforementioned results were obtained by the \emph{absorbing method}, initiated by R\"odl, Ruci\'nski, and Szemer\'edi \cite{RRS06}.  Roughly speaking, the absorbing method reduces the task of finding a spanning sub(hyper)graph to that of finding a near spanning sub(hyper)graph by using some \emph{absorbing structure}. Given a $k$-graph that contains a matching $M$ and a vertex set $S$ such that $V(M)\cap S= \emptyset$, we say that $M$ \emph{absorbs} $S$ if there is another matching $M'$ with $V(M') = V(M)\cup S$. Suppose we want to prove \eqref{eq:md} for some $d<k$. Let $a= \max \left\{ \frac12, 1 - (\frac{k-1}{k})^{k-d} \right\}$. Let $H$ be a $k$-graph with $\d_d(H)\ge (a + 2\g) \binom{n-d}{k-d}$ for some $\g> 0$.
We first apply the following absorbing lemma of H\`an, Person and Schacht~\cite[Lemma 2.4]{HPS}.

\begin{lemma}\cite{HPS}
\label{lem:HPS}
For all $\g>0$ and positive integers $k> d$ there exists $n_0$ such that the following holds for all $n\ge n_0$. Suppose $H$ is a $k$-graph on $n$ vertices with $\d_d(H)\ge (\frac12 + 2\g) \binom{n-d}{k-d}$, then $H$ contains a matching $M$ of size $\g^k n/k$ that can absorb any vertex set $W\subseteq V(H)\setminus V(M)$ with $|W|\in k\mathbb{N}$ and $|W|\le \g^{2k} n$.
\end{lemma}

We next remove $V(M)$ from $H$ and let $H'=H[V(H)\setminus V(M)]$. Then $\d_d(H)\ge (a + \g) \binom{n-d}{k-d}$. If we can show that $H'$ contains a matching that covers all but at most $\g^{2k} n$ vertices, then $M$ can absorb these vertices and we obtain the desired perfect matching of $H$.
Thus it suffices to prove that \emph{every $k$-graph on $n$ vertices with $\d_d(H)\ge \left(a + o(1) \right)  \binom{n-d}{k-d}$ contains a matching that covers all but $o(n)$ vertices}. We call this assertion the \emph{almost perfect matching lemma}, and note that it is weaker than an asymptotic version of Conjecture~\ref{conj:ms} from Section~\ref{ss:2}.
When $d\ge k/2$, this lemma essentially follows from a greedy argument (see \eg, \cite[Theorem 1.3]{HPS}),
but it appears hard to prove the lemma in general when $d < k/2$.
As shown in \cite{AFHRRS, RR-survey}, it suffices to find an almost perfect \emph{fractional matching} instead (see Section~\ref{ss:3} for details). This helps the authors of \cite{AFHRRS} to obtain $m_d(k, n)$ asymptotically for $d\ge k-4$. However, finding an almost perfect fractional matching for smaller $d$ is still an open problem (see Conjecture~\ref{conj:fpm}).

Now suppose we want to find $m_d(k, n)$ \emph{exactly}.
Naturally we separate the \emph{extremal case} (when $H$ is close to the extremal configuration) from the non-extremal case.
The proof for the non-extremal case follows the procedure described above, except that we may assume that $H$ is \emph{not} close to the extremal configuration when proving the absorbing lemma and the almost perfect matching lemma. For example, suppose $1 - (\frac{k-1}{k})^{k-d}< 1/2$, we need the following refinement of Lemma~\ref{lem:HPS}. Let $\eps>0$ and $H$ and $H'$ be two $k$-graphs on the same $n$ vertices. We say that $H$ is \emph{$\eps$-close to $H'$} if $H$ becomes a copy of $H'$ after adding and deleting at most $\eps n^k$ edges. Let $\mathcal B_{n,k}$ ($\overline{\mathcal B}_{n,k}$) denote the $k$-graph whose vertex set can be partitioned into $A, B$ with $|A|=\lfloor n/2 \rfloor$ and $|B|=\lceil n/2 \rceil$ such that all its edges intersect $A$ in an odd (even) number of vertices.
\begin{theorem}\cite[Theorem 5]{TrZh15}
\label{thm:TZ3}
Given any $\eps >0$ and integer $k \geq 2$, there exist $0< \a, \xi < \eps$
and $n_0 \in \mathbb N$ such that the following holds.  Suppose that $H$ is a $k$-graph on $n \geq n_0$ vertices with $\delta _{1} (H) \geq \left( \frac{1}{2}-\a \right) \binom{n - 1}{k- 1}$.
Then $H$ is $\eps$-close to $\mathcal B_{n,k}$ or $\overline{\mathcal B}_{n,k}$, or $H$ contains
a matching $M$ of size $|M| \le \xi n/k$ that absorbs any set $W\subseteq V(H) \setminus V(M)$ such that $|W| \in k\mathbb{N}$ with $|W| \le \xi^2 n$.
\end{theorem}
When $1 - (\frac{k-1}{k})^{k-d}< 1/2$, a general extremal case was solved in \cite[Theorem 4.1]{TrZh12}: \emph{for $1\le d< k$, there exists $n_0$ such that every $k$-graph $H$ on $n\ge n_0$ vertices contains a perfect matching if $H$ is $\eps$-close to $\mathcal B_{n,k}$ or $\overline{\mathcal B}_{n,k}$ and $\d_d(H)\ge \d(n, k, d)+1$}.

 \subsection{Smaller matchings}
\label{ss:2}

In this section we discuss the connection between $\nu(H)$, the size of the largest matching in $H$, and the minimum $d$-degree $\delta_d(H)$. Given $0 \le d< k\le n$ and $s\le n/k$, let $\mathbf{m_d^s(k, n)}$ be the smallest $m$ such that every $k$-graph $H$ on $n$ vertices with $\delta_d(H)\ge m$ contains a matching of size $s$.

When $d=k-1$, R\"odl, Ruci\'nski and Szemer\'edi \cite{RRS06, RRS09} noticed a striking contrast between the codegree threshold for perfect matchings (about $n/2$ as seen in Theorem~\ref{thm:RRS}) and the one for almost perfect matchings (about $n/k$ as shown below). We present the simple proof from \cite{RRS09} because it is a good example of a greedy argument.

\begin{proposition}\cite[Fact 2.1]{RRS09}
\label{prop:RRS}
Given integers $n\ge k\ge 2$, every $k$-graph $H$ with $\d_{k-1}(H)\ge s$ for some $s \le \lfloor n/k \rfloor - k + 2$ contains a matching of size $s$.
\end{proposition}

\begin{proof}
Let $M$ be the largest matching of $H$ and $U= V(H)\setminus V(M)$. Suppose to the contrary, that $|M|< s$ and thus $|V(M)|\le k(\lfloor n/k \rfloor - k + 1)\le n - (k-1)k$. Hence $|U|\ge (k-1)k$. Let $S_1, \dots, S_k$ denote (arbitrary) $k$ disjoint $(k-1)$-subsets of $U$. By the minimum codegree condition, we have $\sum_{i=1}^k \deg(S_i)\ge ks$. All the neighbors of $S_i$ are in $V(M)$ -- otherwise we can enlarge $M$. Since $|M|< s$, there exists an edge $e\in M$ that contains at least $k+1$ neighbors of $S_1, \dots, S_k$. Consequently there are two vertices $v_1, v_2\in e$ and $i\ne j$ such that $
v_1\in N(S_i)$ and $v_2\in N(S_j)$. By replacing $e$ with $\{v_1\}\cup S_i$ and $\{v_2\}\cup S_j$, we obtain a matching of size $|M|+1$, contradiction.
\end{proof}

When $k$ does not divide $n$, the largest matching in a $k$-graph on $n$ vertices is of size $\lfloor n/k \rfloor$. Such matching is called a \emph{near perfect} matching. R\"odl, Ruci\'nski and Szemer\'edi \cite{RRS09} proved that every $k$-graph $H$ on $n$ vertices with $\d_{k-1}(H)\ge n/k + O(\log n)$ contains a near perfect matching and conjectured that $\d_{k-1}(H)\ge \lfloor n/k \rfloor$ suffices.
Using the absorbing method, Han \cite{Han14} recently proved this conjecture.
\begin{theorem}\cite{Han14}
\label{thm:Han1}
Let $n\ge k\ge 2$ be integers such that $k \nmid n$ and $n$ is sufficiently large. Then every $k$-graph $H$ on $n$ vertices with $\d_{k-1}(H)\ge \lfloor n/k \rfloor$ contains a matching of size $\lfloor n/k \rfloor$.
\end{theorem}

A corollary of Theorem~\ref{thm:Han1} is $m^s_{k-1}(k, n)\le s$ for \emph{all} $s< n/k$ and sufficiently large $n$, which generalizes Proposition~\ref{prop:RRS}. To see this, let $H=(V, E)$ be a $k$-graph on $n$ vertices with $\d_{k-1}(H)\ge s$ for some $s< n/k$. By adding a set $T$ of about $\frac{n - ks}{k-1}$ vertices to $V$ and all the $k$-subsets of $V\cup T$ that intersect $T$  to $E$, we obtain a $k$-graph $H'$
on $n'$ vertices with $\d_1(H)\ge \lfloor n'/k \rfloor$ and $n'\not\in k\mathbb{N}$. Applying Theorem~\ref{thm:Han1}, we obtain a near perfect matching of $H'$. Removing the vertices of $T$ from this matching, we obtain a matching in $H$ of size at least $s$.

For any integer $s\le  n/k$, the $k$-graph $H^0_s(n, k)$ defined in Construction~\ref{cons:ms} satisfies $\nu(H^0_s(n, k))=s-1$ and $\delta_d(H)= \binom{n-d}{k-d}-\binom{ n-s-d +1}{k-d}$ for all $0\le d\le k-1$. This implies that
\begin{equation}\label{eq:H0}
    m_d^s(k, n) \ge \binom{n-d}{k-d}-\binom{ n-s-d +1}{k-d}+1 \quad \text{for all } s\le n/k.
\end{equation}
Combining this with the aforementioned corollary of  Theorem~\ref{thm:Han1},
we obtain that for all $k\ge 2$ and sufficiently large $n$,
\[
m^s_{k-1}(k, n)= s \qquad \text{for all $s< n/k$}.
\]
This prompts us to conjecture that $H^0_s(n, k)$ provides the correct value of $m_d^s(k, n)$ for all $d<k$ and most $s< n/k$. As noted in Section~\ref{ss:1}, an asymptotic version of this conjecture already implies \eqref{eq:md}, which determines $m_d(k, n)$ asymptotically for all $d<k$.

\begin{conjecture}\label{conj:ms}
Given $1\le d\le k-2$, there exist $n_0$ and $C$ such that
\begin{equation}\label{eq:msd}
    m_d^s(k, n) = \binom{n-d}{k-d}-\binom{ n-s-d +1}{k-d}+1
\end{equation}
for all $n\ge n_0$ and all $s\le n/k -C$.
\end{conjecture}

When $d=1$, a result of Bollob\'as, Daykin and Erd\H{o}s \cite{BDE} showed that \eqref{eq:msd} holds for all $n> 2k^3 (s+1)$. 
When $d=1$ and $k=3$, K\"uhn, Osthus and Treglown \cite{KOT} proved \eqref{eq:msd} for all $s\le n/3$ and sufficiently large $n$ (not necessarily divisible by $3$).
K\"uhn, Osthus and Townsend~\cite[Conjecture 1.3]{KOTo} proposed an asymptotic version of Conjecture~\ref{conj:ms},
\[
m_d^s(k, n) = \left(1-\left(1 - \frac{s}{n} \right) ^{k-d} + o(1) \right ) \binom{n-d}{k-d},
\] 
and proved it for all $s\le \min\{ \frac{n}{2(k-d)}, \frac{n - o(n)}{k} \}$.


\subsection{Erd\H{o}s Conjecture and fractional matching}
\label{ss:3}

How many edges of a $k$-graph guarantee a matching of size $s$? This question dates back to an old conjecture of Erd\H{o}s \cite{Erd65mat} that has received much attention lately.
\begin{conjecture}\cite{Erd65mat}
\label{conj:Erd}
Let $s, k, n$ be integers such that $2\le k\le n$ and $1\le s\le n/k$. Then
\[
    m_0^s(k, n) = \max \left\{ \binom{ks-1}{k}, \binom{n}{k}-\binom{ n-s +1}{k} \right\}+1.
\]
\end{conjecture}
The lower bound comes from \eqref{eq:H0} and
the $k$-graph consisting of a complete $k$-graph $K^k_{ks-1}$ on $ks-1$ vertices and $n-ks+1$ isolated vertices.

The $s=2$ case of Conjecture~\ref{conj:Erd} is the well-known Erd\H{o}s-Ko-Rado theorem \cite{EKR}.
A classic theorem of Erd\H{o}s and Gallai \cite{ErGa59} confirms the conjecture for $k=2$. Erd\H{o}s \cite{Erd65mat} proved the conjecture for $n\ge n_0(k, s)$. Bollob\'as, Daykin and Erd\H{o}s \cite{BDE} proved the conjecture for $n> 2k^3 (s-1)$ and Huang, Loh, and Sudakov \cite{HLS12} recently improved it to $n\ge 3k^2 s$.
When $k=3$, Frankl, R\"odl, and Ruci\'nski \cite{FRR12} proved the conjecture for $n\ge 4s$ while {\L}uczak and Mieczkowska \cite{LuMi} proved it for sufficiently large $s$. Recently Frankl \cite{Frankl12} proved the conjecture for $k=3$.  Frankl \cite{Frankl13} also proved the conjecture for $s\le n/(2k)$.

Alon et al. \cite{AFHRRS} considered a fractional version of the Erd\H{o}s conjecture.
Let $H=(V, E)$ be a $k$-graph on $n$ vertices. A \emph{fractional matching} in $H$ is a function $w: E\rightarrow [0,1]$ such that for each $v \in V$ we have $\sum _{e \ni v} w(e)\leq1$. The \emph{size} of $w$, denoted by $\nu^*(H)$, is $\sum _{e \in E} w(e)= \frac1{k}\sum_{v} \sum _{e \ni v} w(e) \le n/k$.
If $\nu^*(H) = n/k$ then we call $w$ a \emph{perfect fractional matching}.  Determining $\nu^*(H)$ is a linear programming problem. Its dual problem is finding a minimum \emph{fractional vertex cover} $\tau^*(H)= \sum_{v\in V} w(v)$ over all functions $w: V\rightarrow [0,1]$ such that $\sum_{v\in e} w(v) \ge 1$ for all $e\in E$. Correspondingly $\tau(H)$ is the minimum number of vertices in a vertex cover of $H$.

Given integers $n\ge k> d\ge 0$ (with $k\ge 2$) and a real number $0< s\le n/k$, define $\mathbf{f^s_d(k, n)}$ to be the smallest integer $m$ such that every $k$-graph $H$ on $n$ vertices with $\delta _d (H) \geq m$ contains a fractional matching of size $s$. We let $f_d(k, n):= f^{n/k}_d(k, n)$ (note that $n/k$ may not be an integer).
The hypergraph $H^0_{\lceil s \rceil}(n,k)$ defined in Construction~\ref{cons:ms} contains no fractional matching of size $s$ because
\[
\nu^*(H^0_{\lceil s \rceil}(n, k))= \tau^*(H^0_{\lceil s \rceil}(n, k)) \le \tau(H^0_{\lceil s \rceil}(n, k))= {\lceil s \rceil}-1<s.
\]
When $s$ is an integer, the complete $k$-graph $K^k_{ks-1}$ contains no fractional matching of size $s$.
When $s= n/k$, it was shown in \cite{RRS06mat} that $f_{k-1}(k, n) = \d_{k-1}(H^0_{\lceil n/k \rceil}(n, k)) + 1 =\lceil n/k \rceil$.  The following two conjectures were given in  \cite{AFHRRS}.
\begin{conjecture}\cite{AFHRRS}
\label{conj:fpm}
\begin{enumerate}[{\rm (i)}]
  \item For all $1\le d\le k-1$,
$\displaystyle
f_d(k, n)= \left(1-\left(\frac{k-1}{k}\right)^{k-d} + o(1) \right) \binom{n-d}{k-d}.
$
  \item For all positive integers $n, k, s$ such that $k\ge 2$ and $s\le n/k$,
  \[
  f^s_0(k, n)= \max \left\{ \binom{ks-1}{k}, \binom{n}{k}-\binom{ n-s +1}{k} \right\}+1.
  \]
\end{enumerate}

\end{conjecture}
Alon et al.  \cite{AFHRRS} first proved Conjecture~\ref{conj:fpm} (ii) asymptotically for $k\in \{3, 4\}$ by reducing it to an old probabilistic conjecture of Samuels \cite{Sam} and then they applied this result to prove (i) with $k-d\in \{3, 4\}$.

One can convert an almost perfect fractional matching to an almost perfect integer matching by applying either a result of Frankl and R\"odl \cite{FrRo85} or the weak Regularity Lemma (see \cite[Section 4]{AFHRRS} and \cite[Section 5]{KOTo} for details of these two approaches). This implies the following connection between $m_d(k, n)$ and $f_d(k, n)$.
\begin{theorem}\cite{AFHRRS}
\label{thm:afh}
For $1\le d\le k-1$ if there exists $c>0$ such that $f_d(k, n)=( c+ o(1)) \binom{n-d}{k-d}$ then
 \begin{equation*}
m_{d}(k, n)= \left( \max \left \{ \frac12, \ c \right \} + o(1) \right) \binom{n-d}{k-d}.
\end{equation*}
\end{theorem}

This theorem implies that to determine $m_d(k, n)$, it suffices to prove Conjecture~\ref{conj:fpm} (i).
Theorem~\ref{thm:afh} was slightly improved in \cite{TrZh15} to
 \begin{equation*}
m_{d}(k, n)= \max \left \{ \delta (n,k,d)+1, \ (c  + o(1)) \binom{n- d}{k- d} \right \},
\end{equation*}
where $\delta(n, k, d)$ was defined in Section~\ref{ss:1}.

\subsection{Other Matching problems}

First we consider matching problems in multipartite hypergraphs. Let $\K_k(n)$ denote the family of all $k$-partite $k$-graphs $H$ with parts $V_1, \cdots, V_k$ such that $|V_1| = \cdots = |V_k|=n$. 
Given $I\subseteq [k]$, a vertex set  $S$ is \emph{$I$-crossing} if $|S\cap V_i|=1$ for $i\in I$ and $S\cap V_i= \emptyset$ otherwise. Let $\d_I(H)$ be the minimum of $\deg(S)$ over all $I$-crossing sets $S$ and let $\d'_d(H) := \min \d_I(H)$ over all $I\subseteq [k]$ of size $d$.
K\"uhn and Osthus \cite{KO06mat} proved that every $H\in \K_k(n)$ contains a perfect matching if $\d'_{k-1}(H)\ge n/2 + \sqrt{2n \log n}$. Aharoni, Georgakopoulos, and Spr\"ussel \cite{AGS} found an elegant proof of the following theorem.
\begin{theorem}
\label{thm:AGS}\cite{AGS}
Let $H\in \K_k(n)$. If $\deg(S)> n/2$ for all $[k-1]$-crossing sets $S$ and $\deg(S)\ge n/2$ for all $\{2, 3, \dots, k\}$-crossing sets $S$, then $H$ contains a perfect matching.
\end{theorem}
This theorem implies that $H$ contains a perfect matching if $\d'_{k-1}(H)> n/2$ for \emph{all} values of $n$. This is tight when $k$ is even and $n=2 \pmod 4$ but may not be tight in other cases (off by at most one).
Pikhurko \cite{Pik} proved that every $H\in \K_k(n)$ contains a perfect matching if there exists nonempty $L\subsetneq [k]$ such that
\[
\frac{\d_L(H)}{n^{k-|L|}} + \frac{\d_{[k]\setminus L}(H)}{n^{|L|}}= 1 + \Omega\left(\sqrt{\tfrac{\log n}{n}}\right).
\]
This implies that $H$ contains a perfect matching if $\d'_d(H)\ge (1/2 + o(1))n^{k-d}$ for some $d\ge k/2$.
Recently Lo and Markstr\"om \cite{LoMa-3p} determined the minimum $\d'_1(H)$ that guarantees a perfect matching when $H\in \K_3(n)$ and $n$ is sufficiently large.

\medskip

Keevash and Mycroft \cite{KeMy1} investigated the structure of $k$-graphs satisfying $\d_{k-1}(H)\ge n/k - o(n)$ but containing no perfect matching. They showed that such $k$-graphs are close to either space or divisibility barriers (see Theorem~\ref{thm:KM} in the next section). Applying the results in \cite{KeMy1}, Keevash, Knox, and Mycroft recently \cite{KKM15} found a necessary and sufficient condition for the existence of a perfect matching in all $k$-graphs $H$ with $\d_{k-1}(H)\ge n/k + o(n)$.

Given $k\ge 3$ and $0\le c\le 1$, let \textbf{PM}$(k, c)$ denote the decision problem of determining whether a $k$-graph $H$ contains a perfect matching when $\d_{k-1}(H) \ge cn$. The result of Karp \cite{Karp} says that \textbf{PM}$(k,0)$ is NP-complete. On the other hand, Theorem~\ref{thm:RRS} implies that
\textbf{PM}$(k, 1/2)$ can be decided in constant time.
Szyma\'nska \cite{Szy13} proved that for $c < 1/k$ the problem \textbf{PM}$(k,0)$ admits a polynomial-time reduction to \textbf{PM}$(k, c)$ and hence \textbf{PM}$(k, c)$ is also NP-complete. Karpi\'nski, Ruci\'nski and Szyma\'nska \cite{KRS10} showed that there exists $\eps > 0$ such that \textbf{PM}$(k, 1/2-\eps)$ is in P.  Applying the aforementioned structure result, Keevash, Knox, and Mycroft \cite{KKM15} proved that \PM$(k, 1/k + \eps)$ is in P for all $\eps>0$ (furthermore, their algorithm provides a perfect matching if it exists). Very recently Han \cite{Han14_Poly} solved the remaining case \PM$(k, 1/k)$.
In fact, he proved that \PM$(k, c)$ is in P for all $c\ge 1/k$ by using some theory developed in \cite{KKM15} and a lattice-based absorbing method.

\section{Packing Problems}

\subsection{Results}
\emph{(Hyper)graph packing}, alternatively called \emph{(hyper)graph tiling}, is a natural extension of the matching problem. Given two $k$-graphs $H$ and $F$, an \emph{$F$-packing} is a sub(hyper)graph of $H$ that consists of vertex-disjoint copies of $F$. An $F$-packing is \emph{perfect} (called an $F$-\emph{factor}) if it is a spanning sub(hyper)graph of $H$. In this case it is necessary that $|V(F)|$ divides $|V(H)|$.

Packing problems have been studied extensively for graphs. The celebrated Hajnal-Szemer\'edi theorem \cite{HaSz} states that for all $n\in t\mathbb{N}$ every $n$-vertex graph with minimum degree at least $(1-1/t)n$ contains a $K_t$-factor.
Given a $k$-graph $F$ of order $f$ and an integer $n$ divisible by $f$, we define \emph{the $F$-packing threshold} $\mathbf{\delta_d(n,F)}$  as the smallest integer $t$ such that every $n$-vertex $k$-graph $H$ with $\d_{d}(H)\ge t$ contains an $F$-factor. We simply write $\d(n, F)$ for $\d_{k-1}(n, F)$.
When $k=2$, K\"{u}hn and Osthus \cite{KuOs09} determined $\d(n,F)$ up to an additive constant for all graphs $F$. This improves the earlier results of Alon and Yuster \cite{AY96}, Koml\'os \cite{Kom}, and Koml\'os, S\'ark\"ozy, and Szemer\'edi \cite{KSS-AY}. For more results on graph packing, see the survey \cite{KuOs-survey}.

It is not surprising that packing problems become harder in hypergraphs. Other than the matching problems mentioned in Section~1, only a few packing thresholds are known. Recall that $K_t^k$ is the complete $k$-graph on $t$ vertices. The first step towards a hypergraph Hajnal-Szemer\'edi theorem is determining $\d(n, K_4^3)$ ($K_4^3$-packing is also interesting because the corresponding Tur\'an problem is a famous conjecture of Tur\'an \cite{Turan}). Czygrinow and Nagle \cite{CzNa} showed that $\d(n, K^3_4)\ge 3n/5 + o(n)$. Keevash and Sudakov observed that $\d(n, K^3_4)\ge 5n/8 + o(n)$. Pikhurko \cite{Pik} proved $3n/4 -2 \le \d(n, K^3_4)\le 0.861 n$; independently Keevash and the author (unpublished, see \cite{KeMy1}) proved that $2n/3 -1\le \d(n, K^3_4)\le 4n/5 + o(n)$. Lo and  Markstr\"om~\cite{LoMa-fa} showed that $\d(n, K^3_4)= 3n/4 + o(n) $ by the absorbing method. Independently and simultaneously Keevash and Mycroft \cite{KeMy1} determined $\d(n, K^3_4)$ \emph{exactly}.
\begin{theorem}\cite{KeMy1}
For sufficiently large $n\in 4\mathbb{N}$,
\[
\d(n, K_4^3) =
\begin{cases}
3n/4 -2 & \text{if $n\in 8\mathbb{N}$}  \\
3n/4 - 1 & \text{otherwise.}
\end{cases}
\]
\end{theorem}

We will elaborate on the approaches used in \cite{LoMa-fa, KeMy1} in the next subsection.
It is desirable to find the packing threshold $\d(n, K_t^k)$ for all $3\le k< t$. However, at present such a hypergraph Hajnal-Szemer\'edi theorem seems out of reach. When $t=k+1$, Lo and Markstr\"om~\cite{LoMa-fa} showed that $\d(n, K_{k+1}^k)\le (1 - 1/2k)n$ for $k\ge 3$. It is plausible that one can prove $\d(n, K^k_{k+1}) \le \frac{k}{k+1}n + o(n)$ by applying the approach of \cite{KeMy1}. Unfortunately we do not know a matching lower bound (it was shown in \cite{LoMa-fa} that $\d(n, K_{k+1}^k)\ge 2n/3$ for even $k$).
For arbitrary $t$, it was shown in \cite{LoMa-fa} that
\[
\left( 1 - \frac{193\log (t-1)}{(t-1)^2}\right)n \le \d(n, K_t^3) \le \left( 1 - \frac{2}{t^2 - 3t + 4} + o(1) \right)n
\]
and $\d(n, K_t^k)\le (1 - \binom{t-1}{k-1}^{-1} + o(1))n$ for $k\ge 6$ and $t\ge (3+\sqrt{5})k/2$.
\begin{problem}
\begin{enumerate}
  \item Prove or disprove that $\d(n, K^k_{k+1}) = \frac{k}{k+1}n + o(n)$ for $k\ge 4$.
  \item Improve the existing bounds for $\d(n, F)$ when $F=K_t^3$, $t\ge 5$ and $F=K_t^k$, $4\le k<t$.
\end{enumerate}
\end{problem}

Let $K_4^-$ be the (unique) 3-graph with four vertices and three edges. Lo and Markstr\"om~\cite{LoMa-k4} showed that $\d(n, K^-_4)= n/2 + o(n) $ by using the absorbing method. 

All the remaining success was on packing with $k$-partite $k$-graphs. Given positive integers $m_1\le \cdots \le m_k$, let $K^k_{m_1, \dots, m_k}$ denote the complete $k$-partite $k$-graph with parts of sizes $m_1, \dots, m_k$. In particular, let $K^k_k(m)= K^k_{m, \dots, m}$. It is clear that $\d_d(n, K^k_k(m))\ge m_d(k, n)$ but it is possible to have $\d_d(n, K^k_{m_1, \dots, m_k})< m_d(k, n)$ for certain $m_1, \dots, m_k$.
Other than the matching problems, perhaps the earliest result on hypergraph packing was on $K^3_{1,1,2}-packing$ (note that $K^3_{1,1,2}$ is the unique 3-graph with four vertices and 2 triples).
As a corollary of their main result on loose Hamilton cycles, K\"{u}hn and Osthus \cite{KuOs-hc} proved that $\d(n, K^3_{1,1,2})= n/4+ o(n)$. Recently Czygrinow, DeBiasio, and Nagle \cite{CDN} determined this threshold exactly for sufficiently large $n$.

Mycroft \cite{Myc14} recently determined the codegree packing thresholds asymptotically for all complete $k$-partite $k$-graphs, as well as a large class of non-complete $k$-partite $k$-graphs.
Given a complete $k$-partite $k$-graph $K= K^k_{m_1, \dots, m_k}$, define $\gcd(K)$ to be $\gcd\{ m_j - m_i: 1\le i<j\le k\}$ and $\sigma(K)= m_1/(m_1 + \dots + m_k)$. We say that $K$ is type 0 if $\gcd(m_1, \dots, m_k)> 1$ or all $m_i =1$; $K$ is type $d\ge 1$ if $\gcd(m_1, \dots, m_k)= 1$ and $\gcd(K)=d$. It was shown in \cite{Myc14} that
\begin{equation*}
\d(n, K) =
\begin{cases}
n/2 + o(n) & \text{if $K$ is type 0;}  \\
\sigma(K)n + o(n) & \text{if $K$ is type 1;} \\
\max \{ \sigma(K)n, n/p \} + o(n) & \text{if $K$ is type $d\ge 2$,}
\end{cases}
\end{equation*}
where $p$ is the smallest prime factor of $d$. Mycroft \cite{Myc14} also answered a question of R\"odl and Ruci\'nski \cite{RR-survey} by determining $\d(n, C^k_s)$ asymptotically for all $k\ge 3$ and $s\ge 2$, where $C^k_s$ is a $k$-uniform loose cycle with $s$ edges (see Section~\ref{sec:hc} for its definition). The proof of \cite{Myc14} makes use of hypergraph regularity, including the Regularity Approximation Lemma of R\"odl and Schacht \cite{RS} and the Blow-up Lemma of Keevash \cite{Ke-bu}.

Let us consider hypergraph packing under vertex degree conditions.\footnote{We are not aware of any packing results on $d$-degree conditions for $1<d<k-1$.} Very little is known beyond the matching problem. Lu and Sz\'ekely \cite{LuSz} used the Local Lemma to derive a general upper bound for $\d_1(n, F)$ for arbitrary $k$-graph $F$.
\begin{theorem}\cite{LuSz}
\label{thm:LS}
Let $F$ be a $t$-vertex $m$-edge $k$-graph in which each edge intersects at most $d$ other edges. Then \begin{equation*}
    \d_1(n, F) \le \left( 1 - \frac{1}{e(d+1+ \frac{m}{t}k^2)} \right)\binom{n-1}{k-1}, \quad
    \text{where } e= 2.718....
\end{equation*}
\end{theorem}
For example, when $F=K^k_t$, we have $m=\binom{t}{k}$, $d< k\binom{t-1}{k-1}$ and consequently $ \d_1(n, K_t^k) \le \left(1- \frac{1}{2ek \binom{t-1}{k-1}}\right)\binom{n-1}{k-1}$. On the other hand, we have $\d_1(n, K_t^k)> \binom{n-1}{k-1} - \binom{(k-1)n/t}{k-1}$ by considering $H^0_{(t- k+1)n/t}(n, k)$ defined in Construction~\ref{cons:ms}.
Therefore, when $k$ is fixed and $t\to \infty$, we have
\begin{equation}\label{eq:d1K}
 \d_1(n, K_t^k) =  \left(1- \Theta \left(\frac{1}{t^{k-1}}\right)\right)\binom{n-1}{k-1}.
\end{equation}

By considering complement $r$-graphs, we can translate \eqref{eq:d1K} to the following corollary on equitable colorings of hypergraphs. An \emph{equitable $\ell$-coloring} of a (hyper)graph is a proper vertex coloring with $\ell$ colors such that the sizes of any two color classes differ by at most one.
\begin{corollary}\label{cor}
For every $k\ge 3$ there exists $c_k>0$ such that every $k$-graph with maximum vertex degree $d$ has an equitable $\ell$-coloring for any $\ell\ge c_k d^{1/(k-1)}$.
\end{corollary}

\begin{problem}
Improve the constant $c_k$ in Corollary~\ref{cor}, \ie, improve the constants hidden in \eqref{eq:d1K}.
\end{problem}

Lo and Markstr\"om \cite{LoMa-fa} applied the results in \cite{khan1,khan2} to determine $\d_1(n, K^3_3(m))$ and $\d_1(n, K^4_4(m))$ asymptotically. Recently Han and the author \cite{HZ3} and independently Czygrinow \cite{Czy14} determined $\d_1(n, K^3_{1, 1, 2})$ exactly for sufficiently large $n$. More recently Han, Zang, and the author \cite{HZZ_tiling} determined $\d_1(n, K^3_{a, b, c})$ asymptotically for arbitrary $a\le b\le c$.  Given $a\le b\le c$, let $d=\gcd(b-a,c-b)$ and define
\begin{equation}\label{absbound}
f(a,b,c) :=
\begin{cases}
1/4, & \text{if $a=1$, $\gcd(a,b,c)=1$ and $d=1$;}\\
6-4\sqrt{2}\approx 0.343, & \text{if $a\geq 2$, $\gcd(a,b,c)=1$ and $d=1$;}\\
4/9 , & \text{if $\gcd(a,b,c)=1$ and $d\geq 3$ is odd;}\\
1/2, & \text{otherwise.}
\end{cases}
\end{equation}

\begin{theorem}\cite{HZZ_tiling}
\label{thm:HZZ}
\begin{equation*}
\d_{1}(n, K_{a,b,c})= \left(\max\left\{f(a,b,c), 1-\left(\frac {b+c}{a+b+c} \right)^2, \left(\frac {a+b}{a+b+c} \right)^2 \right\} + o(1) \right)\binom n2.
\end{equation*}
\end{theorem}

It is interesting to note that Theorem~\ref{thm:HZZ} contains a case where the coefficient of the packing threshold is \emph{irrational}. In fact, as far as we know, all the previously known tiling thresholds have rational coefficients.
The lower bound in Theorem \ref{thm:HZZ} follows from six constructions: three of them are {divisibility barriers}, two are {space barriers}, and the last construction was called a \emph{tiling barrier} because there exists a vertex that is not contained in any copy of $K_{a, b, c}$. In general, given a $k$-graph $F$, let $\tau_d(n, F)$ denote the minimum integer $t$ such that every $k$-graph $H$ of order $n$ with $\delta_d(H) \ge t$ has the property that every vertex of $H$ is covered in some copy of $F$. When $F$ is a graph, it is not hard to see that $\tau_1(n, F) = (1 - 1/(\chi(F) - 1) + o(1))n$ (see the concluding remarks of \cite{HZZ_tiling}).
Given a $k$-graph $F$, trivially
\begin{equation}
\label{eq:tau}
 \ex_d(n, F) < \tau_d(n, F) \le \d_{d}(n, F),
\end{equation}
where $\ex_d(n, F)$ is the \emph{$d$-degree Tur\'an number} of $F$, defined as the smallest integer $t$ such that every $r$-graph $H$ of order $n$ with $\delta_d(H) \ge t+1$ contains a copy of $F$. It was shown \cite{HZZ_tiling} that $\tau_1(n, K^3_{a,b,c})\le (6-4\sqrt{2}+o(1)) \binom{n}{2}$, and equality holds when $a\ge 2$. It will be interesting to know $\tau_1(n, F)$ for other 3-graphs.

\subsection{Methods}
Most packing thresholds on graphs, \eg, \cite{KSS-AY, KuOs09} were obtained by the regularity method using the Regularity Lemma of Szemer\'edi \cite{Sze} and the Blow-up Lemma of Koml\'os, S\'ark\"ozy, and Szemer\'edi \cite{Blowup}. As the hypergraph versions of these two lemmas are now available, it is possible to attack hypergraph packing problems by the same approach though the proofs become long and technical -- \cite{KeMy1, Myc14} are two examples.

Keevash and Mycroft \cite{KeMy1} derived a theorem on perfect matchings for simplicial complexes that is very useful for packing problems. To state this result precisely, we need several definitions.
A \emph{$k$-system} $J$ is a set system in which the largest set of $J$ has size $k$ and $\emptyset \in J$. A \emph{$k$-complex} is a {downward closed} $k$-system, namely, every subset of a set in $J$ is also in $J$.
Given a $k$-system $J$, let $J_r$ denote the family of $r$-sets in $J$ for $0\le r\le k$. The \emph{minimum $r$-degree} of $J$, denoted by $d_r(J)$, is the minimum $\deg_{J_{r+1}}(e)$ among all $e\in J_r$. (Note that this is different from $\d_r(J_{r+1})$, which is the minimum $\deg_{J_{r+1}}(S)$ among all $r$-sets $S\subseteq V(J)$.)
The \emph{degree sequence} of $J$ is $\mathbf{d}(J)= (d_0(J), d_1(J), \dots, d_{k-1}(J))$. Given a vector $\mathbf{a}= (a_0, a_1, \dots, a_{k-1})$, we write $\mathbf{d}(J)\ge \mathbf{a}$ if $d_i(J)\ge a_i$ for $0\le i\le k-1$.

Let $H=(V, E)$ be a $k$-graph and let $\mathcal{P}$ be an ordered partition of $V$ into $V_1, \dots, V_t$. The \emph{index vector} $\mathbf{i}_{\mathcal{P}}(S)$ of a set $S\subseteq V$ is defined as $ ( | S\cap V_1|, \dots, |S\cap V_t| )$. A vector of $\mathbb{Z}^t$ is referred to as a \emph{$k$-vector} if all its coordinates are non-negative and  sum to $k$ (thus all $\mathbf{i}_{\mathcal{P}}(e)$, $e\in E$, are $k$-vectors).
A \emph{lattice} in $ \mathbb{Z}^t$ is an additive subgroup of $ \mathbb{Z}^t$.
We let $L_{\mathcal{P}}(H)$ denote the lattice generated by the index vectors $\mathbf{i}_{\mathcal{P}}(e)$ for all edges $e$ of $H$. Given $\mu>0$, let $L_{\mathcal{P}}^{\mu}(H)$ denote the lattice generated by all vectors $\mathbf{x}\in \mathbb{Z}^t$ such that there are at least $\mu n^k$ edges $e$ of $H$ with $\mathbf{i}_{\mathcal{P}}(e) = \mathbf{x}$. A lattice is \emph{complete} if it contains all $k$-vectors; otherwise it is \emph{incomplete}. A lattice is \emph{transferral-free} if it contains no $\mathbf{u}_i - \mathbf{u}_j$ for any $i\ne j$, where $\mathbf{u}_i$ is the 1-vector whose $i$th coordinate is one.

\begin{theorem}\cite[Theorem 2.9]{KeMy1}
\label{thm:KM}
Suppose that $1/n \ll \a \ll \mu, \b \ll 1/k$ and $k$ divides $n$. Let $J$ be a $k$-complex on $n$ vertices with degree sequence
\begin{equation}
\label{eq:dJ}
\mathbf{d}(J) \ge \left(n, (\tfrac{k-1}{k} - \a)n, (\tfrac{k-2}{k} - \a)n, \dots, (\tfrac{1}{k} - \a)n \right).
\end{equation}
Then at least one of the following properties holds:
\begin{description}
\item[1] $J_k$ contains a perfect matching.
\item[2 (Space barrier)] There exists a set $S\in V(J)$ with $|S|= jn/ k$ for some $1\le j\le k-1$ such that all but at most $\b n^k$ edges of $J_k$ intersect $S$ with at most $j$ vertices.
\item[3 (Divisibility barrier)] There exists a partition $\P$ of $V(J)$ into $t\le k$ parts of size at least $d_{k-1}(J) - \b n$ such that $L_{\mathcal{P}}^{\mu}(J_k)$ is incomplete and transferral-free.
\end{description}
\end{theorem}

After obtaining this theorem, Keevash and Mycroft \cite{KeMy1} attacked the $K_4^3$-packing problem as follows.
Let $H=(V, E)$ be a $3$-graph with $\d(H)\ge 3n/4$. Consider the so-called \emph{clique $4$-complex} $J$ of $H$, in which $J_i= \binom{V}{i}$ for $i\le 2$, $J_3= H$, and $J_4$ is the family of all $4$-sets that span copies of $K_4^3$. It is easy to see that
\[
d_0(J)= n, \quad d_1(J)= n-1, \quad d_2(J)=\d_2(H)\ge \frac34 n, \quad and \quad d_3(J)\ge \frac{n}4 + 3.
\]
Indeed, to find $d_3(J)$, we fix a 3-set $abc\in E(H)$; each of the pairs $ab, ac, bc$ has at least $3n/4 -1$ neighbors in $V\setminus \{a, b, c\}$ and thus at least $n/4+3$ vertices $d\in V$ are the neighbors of all $ab, ac, bc$.\footnote{Here we see why we need to consider the 4-complex $J$ instead of the 4-graph $J_4$ alone: $\d_3(J_4)=0$ because a 3-set $abc\not\in E(H)$ has degree zero in $J_4$.} Next we apply Theorem~\ref{thm:KM} to conclude that either $J_4$ contains a perfect matching or there is a space or divisibility barrier. What remains is to show that $H$ contains a $K_4^3$-factor if it is a space or divisibility barrier satisfying the minimum codegree condition.

Most of the aforementioned packing thresholds were obtained by the absorbing method. As described in Section~\ref{ss:1} for the matching problems, our goal is first obtaining a small absorbing $F$-packing that can absorb any smaller set of vertices, and then finding an $F$-packing that covers most of the remaining vertices.


Given a $k$-graph $F$ of order $f$, suppose we want to find an $F$-factor in an $n$-vertex $k$-graph $H=(V, E)$ with certain degree conditions. Given $\eps>0$, $i\in \mathbb{N}$, and two vertices $x, y\in V(H)$, we say that $x$ and $y$ are \emph{$(\eps, i)$-reachable} if there are at least $\eps n^{if-1}$ $(if-1)$-subsets of $W\subset V$ such that both $H[\{x\} \cup W]$ and $H[\{y\} \cup W]$ contain $F$-factors. A vertex set $U\subseteq V$ is \emph{$(\eps, i)$-closed} if any two vertices of $U$ are $(\eps, i)$-reachable. In order to find the desired absorbing $F$-packing, it suffices to show that $V$ is $(\eps, i)$-closed for some $\eps>0$ and $i\in \mathbb{N}$ -- see \cite[Lemma 1.1]{LoMa-fa}.
Sometimes this is straightforward but sometimes this is done in two steps,
referred to as a \emph{lattice-based absorbing method} in \cite{Han14_Poly, HZZ_tiling}.
In Step 1 we find a partition $\mathcal{P}= (V_1, \dots, V_t)$ of $V$ such that each $V_i$ is not small and $(\eps, i)$-closed for some $\eps>0$ and $i\in \mathbb{N}$. In Step 2 we show that $L_{\mathcal{P}, F}^{\mu}$ is complete for some $\mu>0$, where $L_{\mathcal{P}, F}^{\mu}$ is the lattice generated by all vectors $\mathbf{v}\in \mathbb{Z}^t$ such that there are at least $\mu n^f$ $f$-sets $S$ that span copies of $F$ and satisfy $\iP(S)= \mathbf{v}$. Once these steps are done, we can easily show that  $V$ is $(\eps', i')$-closed for some $\eps'>0$ and $i'\in \mathbb{N}$.


\section{Hamilton Cycles}
\label{sec:hc}
Cycles in hypergraphs have been studied since the 1970s. There are several notions of cycles. A $k$-graph $(V, E)$ is called a \emph{Berge cycle} if $E$ consists of $k$ distinct edges $e_1, \dots, e_t$ and $V$ contains distinct vertices $v_1, \dots, v_t$ (and possibly other vertices) such that each $e_i$ contains $v_i$ and $v_{i+1}$, where $v_{t +1} = v_1$. Bermond, Germa, Heydemann, and Sotteau \cite{BGHS} proved a Dirac-type theorem for Berge cycles. In recent years  a more structured notion of cycles has become more popular. Given $1\le l< k$, a $k$-graph $C$ is a called an \emph{$l$-cycle} if its vertices can be ordered cyclically such that each of its edges consists of $k$ consecutive vertices and every two consecutive edges (in the natural order of the edges) share exactly $l$ vertices. In a $k$-graph, a $(k-1)$-cycle is often called a \emph{tight} cycle while a $1$-cycle is often called a \emph{loose} cycle (sometimes called \emph{linear} cycle). We say that a $k$-graph contains a Hamilton $l$-cycle if it contains an $l$-cycle as a spanning subhypergraph. Note that a $k$-uniform $l$-cycle of order $n$ contains exactly $n/(k - l)$ edges, implying that $k- l$ divides $n$.

We define the threshold $\mathbf{h_d^{l}(k, n)}$ as the smallest integer $m$ such that every $k$-graph $H$ on $n$ vertices with $\d_d(H)\ge m$ contains a Hamilton $l$-cycle, provided that $k- l$ divides $n$. As before, we may omit the subscript when $d=k-1$. Unless stated otherwise, we assume that $n$ is sufficiently large in this section.

Katona and Kierstead \cite{KaKi} first studied $h^{k-1}(k, n)$ and proved that
\[
\left\lfloor \frac{n-k+3}{2} \right\rfloor \le h^{k-1}(k, n) \le \left( 1 - \frac{1}{2k} \right) n + O(1)
\]
The following construction provides the lower bound.
\begin{construction}\label{cons:KK}
Let $H=(V, E)$, where $V=X\cup Y\cup \{v\}$ such that $|X|= \lfloor \frac{n-1}2 \rfloor$ and $|Y|= \lceil \frac{n-1}2 \rceil$, and $E$ consists of all $k$-subsets of $V$ containing $v$ and all $k$-sets $S\subset X\cup Y$ such that $|S\cap X|\ne \lfloor k/2 \rfloor$.
\end{construction}
It is easy to check that $\d_{k-1}(H)\ge \lfloor (n-k+1)/2 \rfloor$. Suppose $H$ contains a tight Hamilton cycle. Then there exists an ordering $v_1, \dots, v_{n-1}$ of the vertices of $X\cup Y$ such that all $e_i:= \{v_i, \dots, v_{i+k-1}\}$, $1\le i\le n-k$, are edges. Let $a_i = | X\cap e_i |$. Then $\sum_{i=1}^{n-k} a_i$ is about $|X| k$ when $n$ is sufficiently large because all the vertices of $X$ are counted $k$ times except for those $v_i$ with $i<k$ and $i> n-k$. Thus the average $a_i$ is about $k/2$. On the other hand, we have $a_i\ne \lfloor k/2 \rfloor$ and $|a_i - a_{i+1}|\le 1$ for all $i$. Thus either all $a_i\le \lfloor k/2 \rfloor -1$ or all $a_i\ge \lfloor k/2 \rfloor +1$, a contradiction.

Using the absorbing method, R\"odl, Ruci\'nski and Szemer\'edi \cite{RRS06, RRS08} showed that
\begin{equation}
\label{eq:hk-1}
h^{k-1}(k, n) = \frac{n}2 + o(n)
\end{equation}
for all $k\ge 3$. With long and involved arguments, they \cite{RRS11} were able to obtain an exact result when $k=3$.
\begin{theorem}\cite{RRS11}
For sufficiently large $n$, $h^{2}(3, n) = \lfloor n/2 \rfloor$.
\end{theorem}

Assume that $1\le l< k$ and $k-l$ divides $n$. Since every $(k-1)$-cycle of order $n$ contains an $l$-cycle on the same vertices, \eqref{eq:hk-1} implies that $h^{l}(k, n) \le \frac{n}2 + o(n)$.
On the other hand, it is not hard to see that $h^{l}(k, n) \ge \frac{n}{2} - k$ when $k-l$ divides $k$. In fact, when $k$ divides $n$, a Hamilton $l$-cycle contains a perfect matching thus Construction~\ref{cons:md} provides this bound; when $k$ does not divide $n$, this was proven by Markstr\"om and Ruci\'nski \cite[Proposition 2]{MaRu}.
Consequently,
\begin{equation}
\label{eq:RRSc}
h^{l}(k, n) = \frac{n}2 + o(n) \quad \text{if } k-l \mid k.
\end{equation}
Very recently Han and the author \cite{HZk2} determined the $d$-degree threshold $h^{k/2}_d(k, n)$ exactly for all even $k\ge 6$ and all $d\ge k/2$. The value of $h_d^{k/2}(k, n)$ turns out to be close (but not always equal) to $\delta(n, k, d)$ defined in Section~\ref{ss:1}.

When $k-l$ does not divide $k$, the threshold $h^l(k, n)$ is much smaller. K\"uhn and Osthus \cite{KuOs-hc} proved that $h^{1}(3, n) = n/4 + o(n)$. This was generalized to arbitrary $k$ and $l=1$ by Keevash, K\"uhn, Mycroft, and Osthus \cite{KKMO} and to arbitrary $k$ and arbitrary $l< k/2$ by H\`{a}n and Schacht \cite{HaSc}. Later K\"uhn, Mycroft, and Osthus \cite{KMO} showed that
\begin{equation}
\label{eq:KOc}
h^{l}(k, n) = \frac{n}{\lceil \frac{k}{k-l} \rceil (k-l)} +o(n) \quad \text{if } k-l \nmid k.
\end{equation}
The following simple construction supports the lower bound in \eqref{eq:KOc}.
\begin{construction}\label{cons:KMO}
Suppose $1\le l< k$, $k-l \nmid k$, and $k-l \mid n$. Let $t= n/(k-l)$ and $s= \lceil {k}/(k-l) \rceil$.
Let $H_0=(V, E)$ be an $n$-vertex $k$-graph in which $V$ is partitioned into sets $A$ and $B$ such that $|A| = \left\lceil t/s \right\rceil - 1$. The edge set $E$ consists of all the $k$-sets that intersect $A$.
\end{construction}
We have $\d_{k-1}(H_0)= |A|= \left\lceil {n}/ \left(\lceil \frac{k}{k-l} \rceil (k-l) \right) \right\rceil - 1$. If $H_0$ contains a Hamilton $l$-cycle $C$, then each vertex is contained in at most $s$ edges of $C$. Since $A$ is a vertex cover of $C$, $|C|\le |A| s< t$, a contradiction.

Recently Czygrinow and Molla \cite{CzMo} showed that $h^{1}(3, n)= \lceil n/4 \rceil$. Independently Han and the author \cite{HZ2} proved that $h^{l}(k, n)= \lceil n/(2k - 2l) \rceil$ for all $l< k/2$. It was conjectured \cite{HZ2} that
Construction~\ref{cons:KMO} is an extremal configuration for Hamilton $l$-cycles whenever $k-l$ does not divide $k$. On the other hand, since the extremal cases in \cite{HZk2, RRS11} require involved work, it seems harder to determine the exact value of $h^{l}(k, n)$ when $k-l$ divides $k$.

\begin{problem}
Determine $h^{l}(k, n)$ exactly for all $l\ge k/2$, in particular, prove that $h^{l}(k, n) = \left\lceil {n}/ \left(\lceil \frac{k}{k-l} \rceil (k-l) \right) \right\rceil$ when $k-l \nmid k$.
\end{problem}

Much less is known on the value of $h_d^{l}(k, n)$ when $d\le k-2$.
Bu\ss, H\`{a}n, and Schacht \cite{BHS} showed that $h_1^{1}(3, n) = (\frac{7}{16} + o(1)) \binom{n}{2}$. Recently Han and the author \cite{HZ1} improved this to an exact result.
\begin{theorem}\cite{HZ1}
There exists $n_0$ such that the following holds. If $H$ is a 3-graph $H$ on $n\ge n_0$ vertices with $n\in 2\mathbb{N}$ and $\delta_1(H)\ge \binom{n-1}2 - \binom{\lfloor\frac34 n\rfloor}2 + c$,  where $c=2$ if $n\in 4\mathbb N$ and $c=1$ otherwise. Then $H$ contains a loose Hamilton cycle.
\end{theorem}

It is conjectured in \cite{BHS} that $h_1^{1}(k, n) = \delta_1(H_0) + o(n^{k-1})$ for $H_0$ defined in Construction~\ref{cons:KMO}. 
\begin{problem}
Determine, asymptotically or exactly, the other values of $h_d^{l}(k, n)$, in particular, prove or disprove that $h_1^1(k, n) = \delta_1(H_0) + o(n^{k-1})$ for all $k\ge 4$.
\end{problem}
It was conjectured in \cite{RR-survey} that $ h_d^{k-1}(k, n) = m_d(k, n) + o(n^{k-d})$ for all $d<k$ and the aforementioned result \cite{RRS09} confirmed this for $d=k-1$.
However, Han and the author \cite{HaZh15} recently disproved this conjecture.
Nevertheless, it might be true that $h_1^{2}(3, n) \approx m_1(3, n) = (\frac59 + o(1))\binom{n}{2}$.
This problem seems (much) harder than the corresponding matching problem. The best known bound $h_1^{2}(3, n) \le \frac13 (5 - \sqrt{5} + o(1)) \binom{n}{2}$ was given by R\"odl and Ruci\'nski \cite{RoRu14} recently. For arbitrary $d<k$, Glebov, Person, and Weps \cite{GPW} proved that $h_d^{k-1}(k, n) \le (1- c_{k, d}) \binom{n-d}{k-d}$ for some small $c_{k, d}>0$.

\medskip
As with the matching and packing problems, the absorbing method is the main tool of finding a Hamilton $l$-cycle in $k$-graphs. Since \cite[Section 2.2]{RR-survey} gave a detailed sketch on this approach (in the case of $l=d= k-1$), we only highlight the main ideas. In a $k$-graph, an \emph{$l$-path} consists of vertices $v_0, v_1, \dots, v_{t-1}$ for some $t= (k-l)q + l$ and edges $E_0, E_1, \dots, E_{q-1}$, where $E_i= v_{(k-l)i} \dots v_{(k-l)i+k-1}$.  We call $v_0 v_1 \dots v_{l-1}$ and $v_{t-l} \dots v_{t-1}$ two \emph{ends} of the path. Given an $l$-path $P$ and a vertex set $S$ such that $V(P)\cap S= \emptyset$, we say that $P$ \emph{absorbs} $S$ if there is an $l$-path on $V(P)\cup S$ with the same ends of $P$. To find a Hamilton $l$-cycle in a $k$-graph $H$, we proceed in the following steps.
\begin{description}
\item[Step 1] Find an \emph{absorbing $l$-path} $P$ with $|P|\ll n$; namely, $P$ can absorb any vertex set $W\subseteq V(H)\setminus V(P)$ such that $|W|\ll |P|$ and $|W|\in (k-l)\mathbb{N}$. Denote the two ends of $P$ by $L_1$ and $L_2$ and let $V' = (V\setminus V(P))\cup L_1\cup L_2$.
\item[Step 2] Find a \emph{reservoir set} $R\subset V'$ of size $|R|\ll |P|$ such that any two $l$-subsets of $V'$ can be connected to an $l$-path via many $s$-subsets of $R$ for some constant $s$ (\ie, independent of $n$). Let $V''= V\setminus (V(P)\cup R)$.
\item[Step 3] Cover all but $o(n)$ vertices of $V''$ with constant many vertex-disjoint $l$-paths.
\end{description}
Once these three steps are done, we connect all the $l$-paths that we have found by using the vertices of $R$ and finally absorb the leftover vertices in Step 3 and the remaining vertices in $R$ with $P$.
If $\d_d(H)\ge m_d(k, n) + o(n^{k-d})$, then Step 3 can be easily done by first applying the weak Regularity Lemma and then finding an almost perfect matching in the reduced $k$-graph. (Thus the difficulty of proving $h^2_1(3, n)\le \frac59\binom n2 + o(n^2)$ resides on the first two steps.) On the other hand, when $l< k-1$, a smaller $\d_d(H)$ may suffice for Step 3. For example, when $2l< k$ and $\d_{k-1}(H)\ge n/(2k - 2l)$, the authors of \cite{HZ2} accomplished Step 3 by finding an almost $Y_{k, 2l}$-factor in the reduced $k$-graph, where $Y_{k, 2l}$ consists of two $k$-sets that share exactly $2l$ vertices. This is possible because one can convert a copy of $Y_{k, 2l}$ in the reduced $k$-graph to an almost spanning $l$-path in the original $k$-graph, due to the fact that, when $2l< k$, every (long) $l$-path is a $k$-partite $k$-graph whose first $2l$ parts are of size about $m$ and the remaining parts are of size about $2m$ for some integer $m$.

\section*{Acknowledgements}
The author would like to thank Jie Han for valuable discussion when preparing this manuscript. He also thanks Albert Bush, Jie Han, Allan Lo, Richard Mycroft, and Andrew Treglown for their comments that improved the presentation of the manuscript.
\bibliographystyle{abbrv}
\bibliography{May2015b}

\end{document}